\documentclass{aptpub}

\authornames{DAVID STENLUND}
\shorttitle{On the Mabinogion urn model}

\usepackage{graphicx}

\DeclareMathOperator{\ZZ}{\mathbb{Z}}
\DeclareMathOperator{\NN}{\mathbb{N}}
\DeclareMathOperator{\PP}{\mathbf{P}}
\DeclareMathOperator{\EE}{\mathbf{E}}

\newcommand{\nbd}{\nobreakdash-\hspace{0pt}}

\begin{document}

\title{On the Mabinogion urn model}

\author[\r{A}bo Akademi University]{David Stenlund}

\address{Mathematics and Statistics, Faculty of Science and Engineering, \r{A}bo Akademi University, FIN-20500 \r{A}bo, Finland}

\email{david.stenlund@abo.fi}

\begin{abstract}
In this paper we discuss the Mabinogion urn model introduced by D. Williams in \emph{Probability with Martingales} (1991). Therein he describes an optimal control problem where the objective is to maximize the expected final number of objects of one kind in the Mabinogion urn model. Our main contribution is formulas for the expected time to absorption and its asymptotic behavior in the optimally controlled process. We also present results for the non-controlled Mabinogion urn process and briefly analyze other strategies that become superior if a certain discount factor is included. 
\end{abstract}



\section{Introduction}

The \emph{Mabinogion urn model} is defined in the following way:
\begin{itemize} \itemsep0em
\item There is an urn containing white and black balls. 
\item At each moment of time $1,2,3,\dotsc$ one ball is drawn randomly from the urn, independently of the previous selections. The color of the ball is observed and the ball is returned to the urn. 
\item If the drawn ball is white, one black ball (if there are any left) is replaced by a white one. Similarly, if the drawn ball is black, one of the remaining white balls (if any) is replaced by a black ball. 
\item If all remaining balls are of the same color the process stops. 
\end{itemize}

Let $\{(W_n,B_n):n=0,1,\dotsc\}$ be the stochastic process corresponding to the Mabinogion urn model, hereafter called the \emph{M\nbd process}, where $W_n$ and $B_n$ are the number of white and black balls, respectively, after $n$ transitions. This process is a Markov chain on a finite state space with two absorbing states (\emph{i.e.} when all balls are of the same color). Thus the hitting time of either of the absorbing states,
\begin{equation*}
H:=\min\{n:W_n=0 \textrm{ or } B_n=0\},
\end{equation*}
is almost surely finite. In Figure~\ref{fig_sample_paths} is a plot of $B_n$ for a few sample paths of the M\nbd process starting from $W_0=B_0=100$, and we can see that sooner or later the process is absorbed in either $B_H=0$ or $B_H=200$. A natural question arising is how long time this is likely to take. The main focus of this paper is indeed on the expected time to absorption, for which we will derive exact formulas as well as asymptotic results. This is done both for the M\nbd process and for a controlled version, which is described below. 

\begin{figure}
\begin{center}
\includegraphics[height=0.3\textwidth]{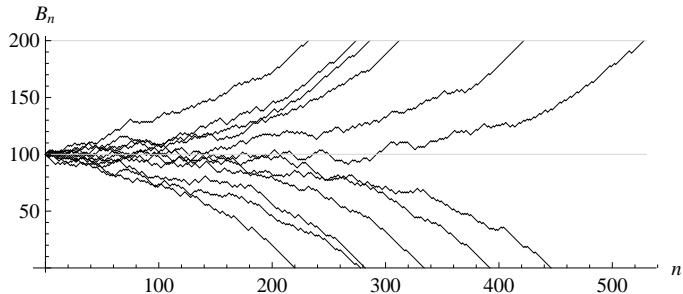}
\caption{Example of possible sample paths for the M-process.}
\label{fig_sample_paths}
\end{center}
\end{figure}

The Mabinogion urn model has for instance been suggested as a model for transmission of radiation damage \cite[pp.~227--230]{Reid} and as a model for a presidential election campaign between two candidates \cite{Flajolet}. The model was first described under the name ``Mabinogion'' by David Williams in his book \emph{Probability with Martingales} (1990) \cite[pp.~159--163]{Williams}. The name comes from a collection of Welsh medieval tales called \emph{The Mabinogion}, where in one of the stories \cite[pp.~144--145]{Mabinogion} there is a herd of sheep that change color like the balls in the urn model. Williams formulated an optimal control problem which will be referred to as the Mabinogion sheep problem, although we here consider colored balls rather than sheep. 

The Mabinogion sheep problem goes as follows. Suppose that the M\nbd process can be adjusted by permanently removing any number of balls from the urn just after time 0 and just after any time $1,2,\dotsc$ when a ball is drawn from the urn. What would be the optimal strategy if the objective is to \emph{maximize the expected final number of black balls}? 

It is obvious that no black balls should be removed during the process, so we can assume that only white balls are removed. Note that for an arbitrary strategy the process is not necessarily a Markov chain, but the process still ends up in an absorbing state almost surely. The optimal strategy is described in \cite{Williams} in the following way. 

At time 0 and after each time a ball has been drawn from the urn:
\begin{itemize} \itemsep0em
\item if the number of white balls is more than or equal to the number of black balls, immediately reduce the number of white balls to one less than the number of black balls; 
\item otherwise, do nothing. 
\end{itemize}
The optimality of this strategy was proved by Williams in \cite{Williams} using martingale theory. We use the same terminology as in \cite{Williams} and call the strategy Policy~A. In this paper we will examine how the expected time to absorption is affected when Policy~A is applied. For the sake of brevity we will call this the \emph{A\nbd controlled M\nbd process}. In particular we will see that when we start from a symmetric initial state we get almost twice as many black balls at the end if Policy~A is applied, while the time to absorption stays on the same level. 

A treatment of some continuous time diffusion models for the Mabinogion sheep problem is found in \cite{Chan}. For instance it is shown that the strategy corresponding to Policy~A is no longer optimal if one considers sufficiently small fractions of balls and also change the transition probabilities accordingly in order to get a suitable weak limit. As noted by Williams, there are similarities between the Mabinogion sheep problem and a portfolio selection problem in \cite{DavisNorman} with two choices of investment (one risky and one risk-free). There the optimal solution is to make a minimal number of transactions in order to keep either of the assets below a certain constant fraction of the total wealth, which can be compared to Policy~A where you keep the number of white balls below $\frac{1}{2}$ of the total number. We refer also to a recent article \cite{Lin}, where a generalization of the Mabinogion sheep problem is considered. 

There are, of course, many other stochastic urn models like the P\'olya urn \cite[pp.~67--68]{Flajolet2}, the OK Corral urn \cite{Kingman,Williams2} or the Ehrenfest urn \cite[pp.~69--71]{Flajolet2} that are somewhat similar to the Mabinogion urn model, but nevertheless have very different behavior. For instance the closely related Ehrenfest urn leads to a reflecting process, whereas the Mabinogion urn leads to an absorbing process. In \cite[Lemma 1]{Flajolet} the authors use a time-reversal transformation of the Ehrenfest urn to derive the probability distribution for the absorption time in the M\nbd process and the expected time to absorption when starting with an equal number of white and black balls, as well as a Gaussian limit for the case when the starting numbers are not equal. Thus there are already some results for the absorption time in the M\nbd process. Our results are instead derived using a similar method as in \cite{Williams}, and for the M\nbd process we thereby get an alternative expression for the expected time to absorption, which is a finite double sum rather than an infinite sum involving either hyperbolic functions or Kac coefficients as in \cite{Flajolet}. Furthermore, the Mabinogion sheep problem is not discussed in \cite{Flajolet} and we especially want to consider Policy~A here. As far as we know these are the first results that compare the time to absorption in both the M\nbd process and the A\nbd controlled M\nbd process. 

In the next section we present the main results, first for the M\nbd process and then for the A\nbd controlled M\nbd process. The results include precise formulas for the expected time to absorption as well as asymptotic expressions when the number of balls in the urn gets very large. We compare the results both in the case when the processes start with an equal number of white and black balls and when these numbers are not equal. We also give a comparison between the M\nbd process and a conditional M\nbd process given the event that there are only black balls left at absorption. In Section~\ref{section_strategies} we give examples of some other strategies for controlling the M\nbd process, as a first step towards general strategy behavior in the Mabinogion sheep problem. A reason for using other strategies might be to have a balance between a large final number of black balls and a short time to absorption. For instance we discuss the optimality when a temporal discount factor is included and conclude that in this case there are better strategies than Policy~A. In Section~\ref{section_proofs} we provide proofs for the main results. One of the lemmas used is a pair of neat binomial identities which are interesting in their own right.

\section{Main results}\label{section_results}

\subsection{The M\nbd process}\label{sub_uncontrolled}

First we consider the M\nbd process as described at the beginning of the introduction. While the main focus of this paper is on the expected time to absorption, we will also include, for the sake of completeness, formulas for the expected final number of black balls, in order to make comparison with earlier results easier. 

Recall that $H:=\min\{n:W_n=0 \textrm{ or } B_n=0\}$ denotes the hitting time of either of the absorbing states. For any $w,b\in\NN$, let 
\begin{equation*}
V(w,b) := \EE(B_H | W_0=w,B_0=b)
\end{equation*}
be the expected final number of black balls and let 
\begin{equation*}
T(w,b) := \EE(H | W_0=w,B_0=b)
\end{equation*}
be the expected time to absorption in the M\nbd process starting with $w$ white and $b$ black balls. First we give an explicit formula for the value of $V(w,b)$. Proofs of the results are postponed to Section \ref{proofs_M}.

\begin{proposition}\label{formula_V0}
For any $w,b\in\NN$, 
\begin{equation} 
V(w,b) = (w+b) 2^{-(w+b-1)} \sum_{i=0}^{b-1}\binom{w+b-1}{i}. \label{eq_V0}
\end{equation}
In particular, for any $k\in\NN$,
\begin{equation} 
V(k,k) = k. \label{eq_V0(k,k)}
\end{equation}
\end{proposition}

\begin{remark}
In Equation \eqref{eq_V0} one can also identify the probability of reaching the absorbing state where only black balls are left, namely
\begin{align}\label{eq_P}
\PP(W_H=0 | W_0=w,B_0=b) &= \PP(B_H=w+b | W_0=w,B_0=b) \nonumber\\
&= 2^{-(w+b-1)} \sum_{i=0}^{b-1}\binom{w+b-1}{i}.
\end{align}
This result is also given in \cite[eq. (4.198)]{Reid}. In the symmetric case when $w=b=k$ it is easy to verify that this probability is 1/2, just as expected. 
\end{remark}

Since the M\nbd process is an absorbing Markov chain on a finite state space we know that the time until absorption has finite expectation $T(w,b)$. The precise value is given in the next formula. 

\begin{proposition}\label{formula_T0}
For any $w,b\in\NN$, 
\begin{equation} 
T(w,b) = \frac{w+b}{2(w+b-1)} \sum_{i=0}^{\min\{w,b\}-1}\sum_{j=i}^{w+b-2-i}\frac{\binom{w+b-1}{i}}{\binom{w+b-2}{j}}. \label{eq_T0}
\end{equation}
In particular, for any $k\in\NN$,
\begin{equation}\label{eq_t0_growth}
T(k,k) = k \sum_{i=0}^{k-1} \frac{1}{2i+1}.
\end{equation}
\end{proposition}

Equation \eqref{eq_t0_growth} leads to the following formula for large $k$. 

\begin{corollary}\label{cor_T(k,k)}
As $k\to\infty$,
\begin{equation}\label{eq_Tkk_growth}
T(k,k)-\frac{k}{2}\left(\ln k + \ln 4 + \gamma \right) \to 0,
\end{equation}
where $\gamma = 0.5772\dots$ is the Euler--Mascheroni constant. 
\end{corollary}

This corollary gives the asymptotic behavior of the expected time to absorption in the M\nbd process when the initial numbers of white and black balls are equal, and was also obtained in \cite[Proposition 2]{Flajolet} via a different approach. Now we turn to the controlled M\nbd process to see how the results above differ when Policy~A is applied.

\subsection{The A\nbd controlled M\nbd process}\label{sub_controlled}

Assume now that the M\nbd process is controlled, and more specifically that we apply Policy~A throughout the process, as described in the introduction. In this case we let $V^A(w,b)$ and $T^A(w,b)$ denote the expected final number of black balls and the expected time to absorption, respectively. In the special case when $w$ and $b$ are equal we use the shorter notations
\begin{equation*}
v_k:=V^A(k,k), \quad t_k:=T^A(k,k). 
\end{equation*}
Furthermore, we introduce
\begin{equation}\label{eq_pk}
p_k := 2^{-2k}\binom{2k}{k}
\end{equation}
since this expression occurs in many of the equations. Note that this value corresponds to the probability of getting $k$ heads out of $2k$ coin tosses. 

For $V^A(w,b)$ we have the following formulas, which are given in \cite[Sections 15.4 and 15.5]{Williams}.

\begin{proposition}\label{prop_Williams}
For any $k\in\ZZ^+$ and $c\in\{1,2,\dotsc,k\}$,
\begin{align} 
&V^A(k-c,k+c) = v_k + (2k-v_k) 2^{-(2k-2)} \sum_{i=k}^{k+c-1}\binom{2k-1}{i}, \label{eq_b1} \\
&V^A(k+1-c,k+c) = v_k + ((2k+1)-v_k) \frac{2^{-(2k-1)}}{1+p_k} \; \sum_{i=k}^{k+c-1}\binom{2k}{i}. \label{eq_b2}
\end{align}
The sequence $(v_k)_{k\geq1}$ satisfies the recursion
\begin{equation}\label{eq_vk_recursion}
v_{k+1} = \frac{1-p_k}{1+p_k}v_k + (2k+1) \frac{2p_k}{1+p_k}, \quad v_1=1.
\end{equation}
When $k\to\infty$,
\begin{equation}\label{eq_vk_approx}
V^A(k,k) - \left(2k+\frac{\pi}{4} -\sqrt{\pi k}\right) \to 0.
\end{equation}
\end{proposition}

Now we turn to the expected absorption time $T^A(w,b)$, for which there are corresponding formulas. We have the following results, which are proved in Section~\ref{proofs_AM}.

\begin{proposition} \label{time_equations}
For any $k\in\ZZ^+$ and $c\in\{1,2,\dotsc,k\}$,
\begin{align} 
&T^A(k-c,k+c) = t_k + (2k\alpha_k(k)-t_k) 2^{-(2k-2)} \sum_{i=k}^{k+c-1}\binom{2k-1}{i}-2k\alpha_k(c), \label{eq_ time1} \\
&T^A(k+1-c,k+c) \nonumber\\
&\qquad = t_k + ((2k+1)\beta_k(k+1)-t_k) \frac{2^{-(2k-1)}}{1+p_k} \; \sum_{i=k}^{k+c-1}\binom{2k}{i} - (2k+1)\beta_k(c), \label{eq_ time2}
\end{align}
where
\begin{equation}\label{eq_alphabeta}
\alpha_k(n) := \frac{1}{2k-1} \sum_{i=k+1}^{k+n-1} \sum_{j=k}^{i-1} \frac{\binom{2k-1}{i}}{\binom{2k-2}{j}}, \quad \beta_k(n) := \frac{1}{2k} \sum_{i=k+1}^{k+n-1} \sum_{j=k}^{i-1} \frac{\binom{2k}{i}}{\binom{2k-1}{j}}.
\end{equation}
The sequence $(t_k)_{k\geq1}$ satisfies the recursion 
\begin{equation}\label{eq_time_recursion}
t_{k+1} = \frac{1-p_k}{1+p_k}t_k + (2k+1)\frac{p_k}{1+p_k} \sum_{i=0}^{k-1}\frac{1}{2i+1},\quad t_1=0. 
\end{equation}
\end{proposition}

Since the recursion formulas \eqref{eq_vk_recursion} and \eqref{eq_time_recursion} look similar, one can also expect a somewhat similar looking expression for the asymptotic growth of the expected time to absorption as in \eqref{eq_vk_approx} above. Indeed, we have: 

\begin{theorem}\label{thm_tk_growth}
As $k\to\infty$,
\begin{equation} \label{eq_tk_growth}
T^A(k,k)-\left(\left(\frac{k}{2}+\frac{\pi}{16}-\frac{\sqrt{\pi k}}{4}\right)\left(\ln k + \ln 4 + \gamma\right)+ \frac{3\pi}{16}-\frac{\sqrt{\pi k}}{4}-\frac{1}{4}\right) \to 0,
\end{equation}
where $\gamma = 0.5772\dots$ is the Euler--Mascheroni constant. 
\end{theorem}

\begin{remark}
In can be noted that equations \eqref{eq_vk_approx} and \eqref{eq_tk_growth} give good approximations for $V^A(k,k)$ and $T^A(k,k)$ already for small values of $k$. The differences to the real values are less than 0.1 for $k>3$ and the relative differences less than 0.001 for $k>25$. 
\end{remark}

From equations \eqref{eq_vk_approx} and \eqref{eq_tk_growth} one can see that for large $k$ the dominating terms in $V^A(k,k)$ and $T^A(k,k)$ are $2k$ and $\frac{k}{2}\ln{k}$, respectively. Comparing this to formulas \eqref{eq_V0(k,k)} and \eqref{eq_Tkk_growth} we have the following result, where we use the standard notation $f(x) \sim g(x)$ to denote that $\lim_{x\to\infty} f(x)/g(x) = 1$. 

\begin{corollary}\label{cor_asym}
As $k\to\infty$,
\vspace{-0.5em}\begin{align*}
V^A(k,k) &\sim 2V(k,k), \\
T^A(k,k) &\sim T(k,k).
\end{align*}
\end{corollary}

Thus, when there is an equal number of white and black balls initially, and this number gets very large, the expected time until the process stops will actually be of the same order of magnitude when using Policy~A as when removing no balls at all, whereas the expected number of final black balls will be close to twice as large when using Policy~A. 

As an illustration of the results in Corollary~\ref{cor_asym}, suppose that we start with 100 000 balls, half of them white and the other half black. If no balls are removed we can expect to get on average 50 000 black balls after a total of 319 582 steps, while using Policy~A the expected values are 99 604 black balls after 318 219 steps. A computer simulation of the A\nbd controlled M\nbd process with said starting conditions and $10^5$ test runs gave an average result of 99 605 black balls and 318 234 steps, in agreement with the expected values given by \eqref{eq_vk_approx} and \eqref{eq_tk_growth}, respectively. 

So far we have focused on the asymptotics in case there are initially equally many white and black balls. Now we will consider the case when they are not equal. It is clear that if $w>b$ then $T^A(w,b)=T^A(b,b)$ by the definition of Policy~A, and so Equation \eqref{eq_tk_growth} holds with $k=b$. On the other hand if $w<b$ then the asymptotic growth may be different. In fact, when $N\to\infty$ and the initial proportion of black balls is greater than $\frac{1}{2}$, the A\nbd controlled M\nbd process is almost surely equal to the non-controlled M\nbd process. This leads to the following result. 

\begin{proposition}\label{tw<b_growth}
Suppose that $b=\lceil xN \rceil$ and $w=\lfloor (1-x)N \rfloor$ for $N\in\ZZ^+$ and some fixed $\frac{1}{2}<x<1$. Then as $N\to\infty$,
\begin{align}
V^A(w,b) &\sim V(w,b) \sim N,\label{eq_Vw<b}\\
T^A(w,b) &\sim T(w,b) \sim \frac{N}{2} \ln\left(\frac{1}{2x-1}\right).\label{eq_Tw<b}
\end{align}
\end{proposition}

Thus, in this case both the expected time to absorption and the expected final number of black balls have an asymptotic growth that is linear in $N$, and the growth rates are the same for both the M\nbd process and the A\nbd controlled M\nbd process. The proof of Proposition~\ref{tw<b_growth} mainly consists of showing that these two processes are a.s. equal when $N\to\infty$, under the given assumptions. Note that the asymptotic growth of $T(w,b)$ in the M\nbd process was already proved in \cite{Flajolet}. 

Table~\ref{table_x} contains values for the mean time to absorption $T^A(w,b)$ obtained when simulating the A\nbd controlled M\nbd process for different values of total initial number of balls $N$ and initial proportion of black balls $x$. Each simulation was repeated 10 000 times. From the table it can be seen that for large $N$ the values in the second column ($x=0.5$) grow more rapidly than the rest. 

A plot of the simulation results is found in Figure~\ref{fig_xgrowth}, with some more points included than in Table~\ref{table_x}. The points in the graph represent the ratio $T^A(w,b)/N$ for five different values of $x$. The lines represent the theoretical asymptotics, as suggested by Theorem~\ref{thm_tk_growth} and Proposition~\ref{tw<b_growth}, and for large $N$ the fit is very good. Figure~\ref{fig_xgrowth} illustrates that for $x=0.5$ the growth is of the order $N \log N$, whereas for $x>0.5$ the growth is linear in $N$ when this $N$ gets large enough. 

\begin{table}
\begin{center}
\caption{Mean time to absorption in simulations for different values of total number of balls $N$ and ratio of black balls $x$ in the initial state of the A\nbd controlled M\nbd process.\vspace{0.5em}}
\label{table_x}
\begin{tabular}{rrrrrr}
\hline
$N$ & $x=0.5$ & $x=0.505$ & $x=0.55$ & $x=0.6$ & $x=0.75$ \\
\hline
200 & 296.77 & 299.59 & 249.67 & 168.91 & 70.07 \\
2000 & 4298.94 & 4246.77 & 2329.69 & 1616.19 & 694.21 \\
20000 & 55349.26 & 48374.06 & 23044.07 & 16099.79 & 6933.57 \\
200000 & 671511.56 & 463353.03 & 230265.43 & 160947.05 & 69318.11 \\
2000000 & 7879981.21 & 4608800.89 & 2302605.35 & 1609446.36 & 693150.51 \\
\hline
\end{tabular}
\end{center}
\end{table}

\begin{figure}
\begin{center}
\includegraphics[height=0.38\textwidth]{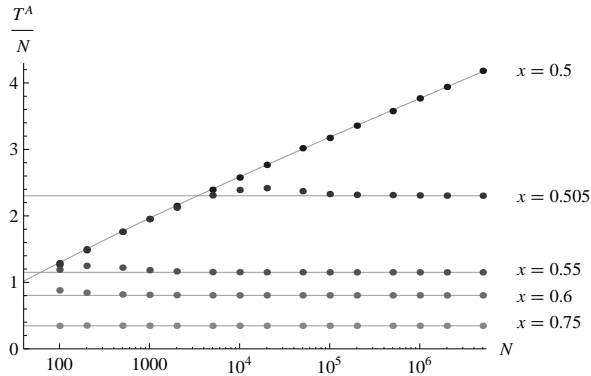}
\caption{Simulated values of $T^A(w,b)/N$ for different $N$ and initial proportion of black balls~$x$. Note the logarithmic scale on the $N$-axis. }
\label{fig_xgrowth}
\end{center}
\end{figure}

\subsection{Conditional M\nbd process}

We can define the M-process conditioned on the event that there are only black balls left at the end. It is natural to consider this conditional process on the state space $S=\{1,2,\dotsc,N\}$. The transition probabilities can be calculated using Bayes' formula. From Equation \eqref{eq_P} we know that the probability of reaching state $N$ before 0 in the M-process is
\begin{equation*}
h(n) := \PP(W_H=0 | W_0=N-n,B_0=n) = 2^{-(N-1)} \sum_{i=0}^{n-1}\binom{N-1}{i}.
\end{equation*}
Note that the function $h$ is harmonic for the M-process killed when it hits state 0 or N, and thus the conditional process can be seen as a Doob $h$-transform of the M-process. Hence the transition probabilities of the conditional M-process are given by
\begin{equation*}
P_{ik}^* := \frac{h(k)}{h(i)} P_{ik},
\end{equation*}
where $P_{ik}$ are the transition probabilities in the M\nbd process, i.e.
\begin{equation*}
P_{ik} = \PP(W_{n+1}=N-k, B_{n+1}=k | W_{n}=N-i, B_{n}=i)=
\begin{cases}
\frac{i}{N}&k=i+1,\\
1-\frac{i}{N}&k=i-1,\\
0&\text{otherwise}.
\end{cases}
\end{equation*}
Let the expected time to absorption in the conditional M\nbd process be 
\begin{equation*}
T_*(w,b) = \EE(H | W_0=w, B_0=b, W_H=0).
\end{equation*}

\begin{proposition}\label{Tstar}
For $k\in\ZZ^+$ it holds that 
\begin{equation*}
T_*(k,k) = T(k,k).
\end{equation*}
If $b=\lceil xN \rceil$ and $w=\lfloor (1-x)N \rfloor$ for $\frac{1}{2}<x<1$ and $N\in\ZZ^+$, then as $N\to\infty$,
\begin{equation*}
T_*(w,b) - T(w,b) \to 0.
\end{equation*}
\end{proposition}

The proof is given in Section \ref{proofs_CM}. In the A\nbd controlled case there is no reason to consider a similar conditional process, since the description of Policy~A ensures that there will always be only black balls left when the process is absorbed.

\section{Outlook on some other strategies in the controlled urn model}\label{section_strategies}

The main results in this paper focus on Policy~A, which has been shown in \cite{Williams} to be optimal when aiming at maximizing the expected final number of black balls. However, the result in Corollary~\ref{cor_asym} shows that the expected time to absorption is asymptotically of the same order as in the (non-controlled) M\nbd process. It would be of interest to find a strategy that gives a significantly smaller expected time to absorption. 

There is an obvious way of minimizing the expected time to absorption, namely to instantly remove all of the white balls. Let us call this strategy Policy~R. Using this strategy the process terminates even before the very first step, so the time is always zero. On the other hand it is not a very good strategy for obtaining many black balls on average. If \emph{both} a large number of black balls \emph{and} a short time are desirable, then a reasonable option is having a strategy that interpolates between the two strategies A and R. In this section we discuss such a family of strategies, and we present some related problems that would be interesting to study further. 

The following is a generalization of Policy~A. Let $q\in(0,1)$ be the cutoff probability characterizing the strategy. If at time $n$ the probability of the next ball drawn being black is less than or equal to $q$ then remove white balls until said probability is greater than $q$, and otherwise do nothing. This strategy will be referred to as a $q$-strategy. The number of white balls removed after step $n$ is equal to $\max(W_n+B_n-\lceil B_n/q \rceil +1,0)$. Note that $q=\frac{1}{2}$ corresponds to Policy~A. 

It should be noted that a couple of $q$-strategies are studied in \cite{Lin}. The author considers a generalization of the Mabinogion urn model, with two additional parameters that determine the probability of a color change when a ball has been drawn from the urn. For two particular choices of these parameters either $q=\frac{1}{3}$ or $q=\frac{2}{3}$ (rather than Policy~A) is shown to be an optimal strategy in the Mabinogion sheep problem. It is an interesting open question how the optimal $q$-value could be found in the general case. However, we do not elaborate further on that question here since we focus on the usual Mabinogion urn model. 

How different $q$-strategies affect the time to absorption is demonstrated in Figure~\ref{fig_qtime}, which shows a plot of the mean absorption time as a function of $q$. The process was simulated $10^6$ times for each of 100 different $q$-values, starting with an equal number of white and black balls. This was done for three different values of total number of balls $N$. Note that the shape of the curve depends on the starting conditions. A larger number of balls in the urn results in a quicker drop for $q>0.5$, as well as a sharper and higher peak when $q$ is slightly below 0.5. It seems that for larger $N$ the peak lies closer to $q=0.5$, and it is an intriguing open problem to know precisely where the maximum is reached and how its value depends on $N$. 

\begin{figure}
\begin{center}
\includegraphics[height=0.38\textwidth]{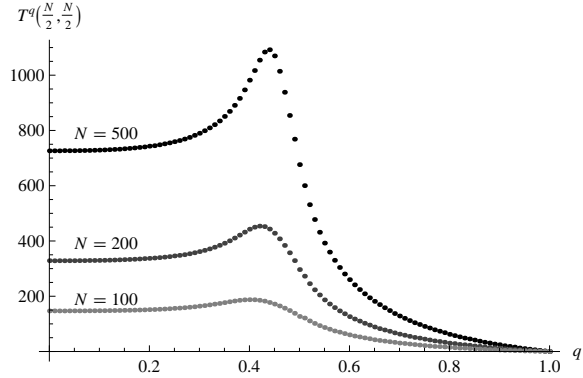}
\caption{Mean time to absorption for different $q$-strategies and three different values of $N$.}
\label{fig_qtime}
\end{center}
\end{figure}

One could imagine a situation where the value of the remaining black balls is decreasing with time, so that the final value at the absorption time H is equal to $e^{-\mu H} B_H$ for some discount factor $\mu$. If $\mu=0$ then Policy~A is the optimal strategy, as has been seen already. If instead $\mu$ gets large enough, then any other strategy will be inferior to Policy~R since even a single step will result in a lower final value than if the process is stopped immediately. For a small positive value of $\mu$, however, one may find a $q$-strategy which is better than both policies A and R. This is illustrated in Figure~\ref{fig_discountvalues} which shows a plot of the final value $e^{-\mu H}B_H$ as a function of the $\mu$ for a number of different $q$-strategies with $q$ between $\frac{1}{2}$ and 1. The values were obtained via simulations of the process starting with 50 white and 50 black balls, where each simulation was repeated $10^6$ times for 20 different values of $\mu$ ranging from 0 to 0.02. 

\begin{figure}
\begin{center}
\includegraphics[height=0.35\textwidth]{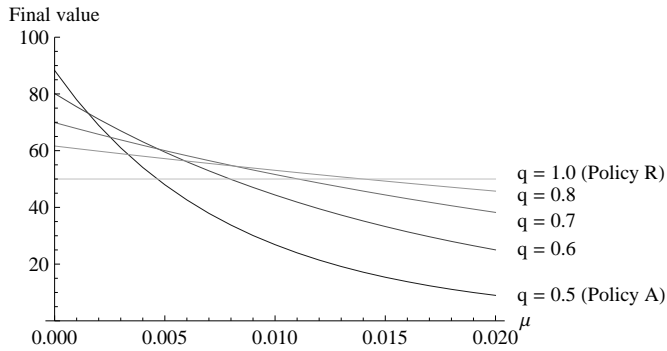}
\caption{Final value as a function of discount parameter $\mu$ using different $q$-strategies.}
\label{fig_discountvalues}
\end{center}
\end{figure}

As can be seen in the figure, initially Policy~A is the best strategy, but as $\mu$ grows the other $q$-strategies clearly become superior. This happens since the smaller average number of final black balls is compensated by a much shorter absorption time. Eventually, however, all strategies fall below the horizontal line corresponding to Policy~R. Judging from the picture it seems that for every $q$-strategy there is at least some value of $\mu$ for which that particular strategy is optimal compared to all the other ones. An interesting open question is therefore to find, given a fixed $\mu$, what value of $q$ gives an optimal strategy that maximizes the expected value of black balls at absorption. 

Using Lemma~\ref{lemma_method} it is possible to derive similar formulas for a $q$-strategy as in Proposition~\ref{prop_Williams} and~\ref{time_equations}. As an example, let the expected final number of black balls under a $q$-strategy be denoted $V^q(w,b)$, and let further $\phi (k) = \lceil k (1-q)/q\rceil$ be the smallest number of white balls for which some have to be removed if there are $k$ black balls present in the urn. When $\frac{1}{2}\leq q<1$ the value $V^q(\phi(k),k)$ satisfies the recursion
\begin{equation}\label{eq_vgamma}
V^q(\phi(k+1),k+1) = \frac{1-p_k^q}{1+p_k^q}V^q(\phi(k),k) + (\phi(k+1) + k) \frac{2p_k^q}{1+p_k^q},
\end{equation}
for all $k\geq 1$, with $V^q(\phi(1),1)=1$ and
\begin{equation*}
p_k^q = \frac{\binom{\phi(k+1)+k-1}{k}}{2\sum_{j=k}^{\phi(k+1)+k-1}\binom{\phi(k+1)+k-1}{j} - \binom{\phi(k+1)+k-1}{k}}. 
\end{equation*}
There are clear similarities between equations \eqref{eq_vk_recursion} and \eqref{eq_vgamma}. Note that when the value $q=\frac{1}{2}$ is inserted in $\phi(k)$ the formulas above coincide with the ones for Policy~A. Similar formulas can be derived for the expected time to absorption under a $q$-strategy. 

Proving Equation \eqref{eq_vgamma} and deriving asymptotic formulas for different $q$-strategies will be the subject of a forthcoming publication.

\section{Proof of the results in Section \ref{section_results}}\label{section_proofs}

\subsection{Key lemmas}\label{proofs_lemmas}

In this section, we state two technical lemmas that will be used extensively in the proofs of our main results. 

\begin{lemma}\label{lemma_method}
Assume that $X(k), k\in\ZZ$ satisfies the following recursion for all $a<k<b$:
\begin{equation}\label{eq_Xrecursion}
X(k) = p(k) X(k-1) + (1-p(k)) X(k+1) + r(k)
\end{equation}
where $p(k)$ and $r(k)$ are functions of $k$, with $p(k)\in(0,1)$. Then for any $a\leq c \leq b$,
\begin{equation}\label{eq_Xc}
X(c) = X(a) -Q(c) + \frac{\displaystyle\sum_{i=1}^{c-a} \prod_{m=1}^{i-1} \frac{p(a+m)}{1-p(a+m)}}{\displaystyle\sum_{i=1}^{b-a} \prod_{m=1}^{i-1} \frac{p(a+m)}{1-p(a+m)}}\big( X(b)-X(a) + Q(b) \big)
\end{equation}
where
\begin{equation}\label{eq_Q}
Q(k) = \sum_{i=1}^{k-a} \sum_{j=1}^{i-1} \frac{r(a+j)}{1-p(a+j)} \prod_{m=j+1}^{i-1} \frac{p(a+m)}{1-p(a+m)}.
\end{equation}
In particular, if $r\equiv 0$ also $Q\equiv 0$. 
\end{lemma}

\begin{proof}[Outline of proof]
The solution method is standard and corresponds for instance to the method used in \cite[pp.~151--156]{KarlinTaylor} for finding the functionals of a general random walk. Writing $Z(k) := X(k)-X(k-1)$ Equation \eqref{eq_Xrecursion} becomes 
\begin{equation*}
Z(k+1) = \frac{p(k)}{1-p(k)}Z(k) - \frac{r(k)}{1-p(k)}.
\end{equation*}
Backwards iteration gives an expression for $Z(k)$ as a function of $Z(a+1)$. Using that $Z(a+1) +\dotsc + Z(b) = X(b)-X(a)$ we can express $Z(k)$ in terms of $X(a)$ and $X(b)$. Forming the sum $Z(a+1) +\dotsc + Z(c)$ then gives the result in Equation \eqref{eq_Xc}. The procedure is not difficult but rather tedious, and not done in detail here.
\end{proof}

Using Lemma~\ref{lemma_method}, it is straightforward to obtain formulas for the expected final number of black balls in the M\nbd process as well as for the expected time until absorption. In the latter case there is a double sum like in Equation \eqref{eq_Q}, but in the special case when the initial numbers of white and black balls are equal this double sum can be much simplified using the following combinatorial results. 

\begin{lemma}\label{lemma_doublesum}
For $n=1,2,3,\dotsc$
\begin{align}
\frac{1}{2n}\sum_{j=0}^{n-1} \sum_{i=0}^{j} \frac{\binom{2n}{i}}{\binom{2n-1}{j}} \;&=\; \frac{1}{2}\sum_{i=0}^{n-1}\frac{1}{2i+1}, \label{eq_doublesum1}\\
\frac{1}{2n-1}\sum_{j=0}^{n-2} \sum_{i=0}^{j} \frac{\binom{2n-1}{i}}{\binom{2n-2}{j}} \;&=\; \frac{1}{2}\sum_{i=0}^{n-1}\frac{1}{2i+1} - \frac{2^{2n-2}}{n \binom{2n}{n}}. \label{eq_doublesum2}
\end{align}
\end{lemma}

\begin{proof}[Outline of proof]
The identities \eqref{eq_doublesum1} and \eqref{eq_doublesum2} can be verified by computer using creative telescoping or the related Wilf--Zeilberger method \cite[Chapters 6--7]{Wilf}. We give an outline of a proof using the hypergeometric function
\begin{equation}\label{eq_hyper}
{}_2F_1(a,b,c;z) := \frac{\Gamma(c)}{\Gamma(a) \Gamma(b)} \sum_{m=0}^\infty \frac{\Gamma(a+m)\Gamma(b+m)z^m}{\Gamma(c+n)\Gamma(m+1)}.
\end{equation}
However, since the calculations are rather tedious not all steps are shown in detail here. For a more detailed proof we instead refer to a forthcoming paper \cite{Stenlund}. 

Let the left hand side of Equation \eqref{eq_doublesum1} be denoted $F(n)$. The idea is to prove the recursion
\begin{equation}\label{eq_F_recursion}
F(n+1)-F(n) = \frac{1}{2(2n+1)}
\end{equation}
for all $n\geq 1$. First we use the identity
\begin{equation}\label{eq_binom_hyper}
\sum_{i=0}^{k} \binom{b}{i} = \frac{1}{2} \binom{b}{k} {}_2F_1\left(1,b+1,b+1-k;\tfrac{1}{2}\right),
\end{equation}
to rewrite $F(n)$ in terms of ${}_2F_1$ as
\begin{equation*}
F(n) = \sum_{j=0}^{n-1} \frac{{}_2F_1\left(1,2n+1,2n+1-j;\tfrac{1}{2}\right)}{2(2n-j)} = \sum_{j=n}^{2n-1} \frac{{}_2F_1\left(1,2n+1,j+2;\tfrac{1}{2}\right)}{2(j+1)}.
\end{equation*}
We then write $F(n+1)-F(n)$ as
\begin{align}
F(n+1)-F(n) &= \sum_{j=n+1}^{2n+1} \frac{{}_2F_1\left(1,2n+3,j+2;\tfrac{1}{2}\right)}{2(j+1)} - \sum_{j=n}^{2n-1} \frac{{}_2F_1\left(1,2n+1,j+2;\tfrac{1}{2}\right)}{2(j+1)} \nonumber\\
&= g(n) - \frac{{}_2F_1\left(1,2n+1,n+2;\tfrac{1}{2}\right)}{2(n+1)} + \frac{2}{2n+1},\label{eq_F1}
\end{align}
where
\begin{equation*}
g(n) = \sum_{j=n+1}^{2n-1} \frac{1}{2(j+1)}\bigg( {}_2F_1\left(1,2n+3,j+2;\tfrac{1}{2}\right) - {}_2F_1\left(1,2n+1,j+2;\tfrac{1}{2}\right)\bigg).
\end{equation*}
The last term in \eqref{eq_F1} is obtained by inserting $b=2n+2$ in \eqref{eq_binom_hyper} twice with $k=0$ and $k=1$, respectively. 

The hypergeometric functions in $g(n)$ are now written as infinite sums like in \eqref{eq_hyper}, then the summands are combined and the order of summation is changed. This gives an inner sum that can be evaluated using a known summation formula \cite[(2.2)]{Gould} so that we get
\begin{align}
g(n) &= \sum_{m=0}^\infty \, \frac{m(4n+3+m)\Gamma(2n+1+m)}{2^{m+1}\Gamma(2n+3)} \left(\sum_{j=n+1}^{2n-1} \,\frac{\Gamma(j+1)}{\Gamma(j+2+m)}\right)\nonumber\\
&= \sum_{m=0}^\infty \, \frac{(4n+3+m)\Gamma(2n+1+m)\Gamma(n+2)}{2^{m+1}\Gamma(2n+3)\Gamma(n+2+m)} - \sum_{m=0}^\infty \, \frac{(4n+3+m)}{2^{m+1}(2n+2)(2n+1)} \nonumber\\
&= \frac{(4n+3){}_2F_1\left(1,2n+1,n+2;\tfrac{1}{2}\right)}{2(2n+2)(2n+1)} + \frac{{}_2F_1\left(2,2n+2,n+3;\tfrac{1}{2}\right)}{4(2n+2)(n+2)} -\frac{2}{2n+1}\nonumber\\
&= \frac{{}_2F_1\left(1,2n+1,n+2;\tfrac{1}{2}\right)}{2(n+1)} - \frac{3}{2(2n+1)} ,\label{eq_F2}
\end{align}
after a number of simplifying steps. The last step follows from the identity
\begin{equation*}
{}_2F_1\left(2,b+1,c+1;\tfrac{1}{2}\right) = \frac{2c}{b}\left((b-2c+2){}_2F_1\left(1,b,c;\tfrac{1}{2}\right) + 2(c-1)\right),
\end{equation*}
which is found by repeated application of Gauss' relations for contiguous hypergeometric functions (see Section 15.2 in \cite{Abram}). 

Combining equations \eqref{eq_F1} and \eqref{eq_F2} gives the recursion in \eqref{eq_F_recursion} and Equation \eqref{eq_doublesum1} follows by induction after checking that it holds for $n=1$. A similar procedure can also be used to prove Equation \eqref{eq_doublesum2} as well as some other similar identities, for which we refer to \cite{Stenlund}. 
\end{proof}

\subsection{M\nbd process; proofs of Proposition~\ref{formula_V0}, Proposition~\ref{formula_T0} and Corollary~\ref{cor_T(k,k)}}\label{proofs_M}

We now turn to the proofs of the results given in Section~\ref{sub_uncontrolled}. From the description of the Mabinogion urn model it is clear that the value of $V(w,b)$ is specified for $w>0$ and $b>0$ by
\begin{equation}\label{eq_a3}
V(w,b) = \frac{w}{w+b} V(w+1,b-1) + \frac{b}{w+b} V(w-1,b+1), 
\end{equation}
and, moreover,
\begin{equation}\label{eq_a1}
V(0,b) = b, \quad V(w,0) = 0. 
\end{equation}
Likewise, it holds for $w>0$ and $b>0$ that
\begin{equation}\label{eq_ta3}
T(w,b) = 1 + \frac{w}{w+b} T(w+1,b-1) + \frac{b}{w+b} T(w-1,b+1),
\end{equation}
while
\begin{equation}\label{eq_ta1}
T(0,b) =  T(w,0) = 0. 
\end{equation}
Note that these two pairs of equations are similar, the differences being that in \eqref{eq_ta1} both boundary values are zero and in \eqref{eq_ta3} there is an additional constant term when compared to \eqref{eq_a3}. In the equations above we now have two recurrence relations \eqref{eq_a3} and \eqref{eq_ta3} with boundary conditions \eqref{eq_a1} and \eqref{eq_ta1}, respectively, which can be solved using Lemma~\ref{lemma_method}. 

\begin{proof}[Proof of Proposition~\ref{formula_V0}]
Note that, in the non-controlled M\nbd process, the total number of balls in the urn, $N$, does not vary with time. We can thus write $w=N-b$, and, for all $0<b<N$, we have
\begin{equation*}
V(N-b,b) = \frac{N-b}{N} V(N-(b-1),b-1) + \frac{b}{N} V(N-(b+1),b+1),
\end{equation*}
with the boundary conditions being $V(N,0)=0$ and $V(0,N)=N$. Applying Lemma~\ref{lemma_method} to this recursion with $p(b)=(N-b)/N$ and $r\equiv 0$ gives after simplification that
\begin{equation*}
V(N-b,b) = N\cdot 2^{-(N-1)}\sum_{i=0}^{b-1} \binom{N-1}{i}.
\end{equation*}
Inserting $N=w+b$ gives Equation \eqref{eq_V0}. When $w$ and $b$ are both equal to $k$ this equation becomes
\begin{equation*} 
V(k,k) = 2k \cdot 2^{-(2k-1)} \cdot 2^{2k-2}= k,
\end{equation*}
proving Equation \eqref{eq_V0(k,k)}.
\end{proof}

\begin{proof}[Proof of Proposition~\ref{formula_T0}]
Using Lemma~\ref{lemma_method} again, we get that
\begin{equation*}
T(N-b,b) = Q(N) 2^{-(N-1)}\sum_{i=0}^{b-1} \binom{N-1}{i} - Q(b),
\end{equation*}
where
\begin{equation*}
Q(b) = \sum_{i=1}^{b} \sum_{j=1}^{i-1} \frac{N}{j} \prod_{m=j+1}^{i-1} \frac{N-m}{m} 
= \frac{N}{N-1} \sum_{i=0}^{b-1} \sum_{j=0}^{i-1} \frac{\binom{N-1}{i}}{\binom{N-2}{j}}.
\end{equation*}
In particular, $Q(N)$ becomes
\begin{align*}
Q(N) &= \frac{N}{N-1} \sum_{i=0}^{N-1} \sum_{j=0}^{i-1} \frac{\binom{N-1}{i}}{\binom{N-2}{j}}
= \frac{N}{N-1}\cdot\frac{1}{2}\left(  \sum_{i=0}^{N-1} \sum_{j=0}^{i-1} \frac{\binom{N-1}{i}}{\binom{N-2}{j}} + \sum_{i=0}^{N-1} \sum_{j=i}^{N-2} \frac{\binom{N-1}{i}}{\binom{N-2}{j}} \right) \\
&= \frac{N}{2(N-1)} \sum_{i=0}^{N-1} \sum_{j=0}^{N-2} \frac{\binom{N-1}{i}}{\binom{N-2}{j}} 
= \frac{2^{N-2}N}{N-1} \sum_{j=0}^{N-2} \frac{1}{\binom{N-2}{j}}
\end{align*}
so the formula for $T(N-b,b)$ can be rewritten as
\begin{align*}
T(N-b,b) &= \frac{N}{2(N-1)} \sum_{i=0}^{b-1}\sum_{j=0}^{N-2} \frac{\binom{N-1}{i}}{\binom{N-2}{j}} - \frac{N}{N-1} \sum_{i=0}^{b-1} \sum_{j=0}^{i-1} \frac{\binom{N-1}{i}}{\binom{N-2}{j}} \\
&= \frac{N}{2(N-1)} \sum_{i=0}^{b-1} \sum_{j=i}^{N-2-i}\frac{\binom{N-1}{i}}{\binom{N-2}{j}},
\end{align*}
where it is assumed that $b\leq N-b$ so that in the inner sum $i \leq N-2-i$. Insertion of $N=w+b$ now proves Equation \eqref{eq_T0} when $b \leq w$. For the other case we note that $T(w,b)=T(b,w)$ due to symmetry, so using the same method we see that \eqref{eq_T0} holds also when $w \leq b$. 

Finally inserting $w=b=k$ in \eqref{eq_T0} gives that
\begin{align*}
T(k,k) &= \frac{2k}{2(2k-1)} \sum_{i=0}^{k-1}\sum_{j=i}^{2k-2-i}\frac{\binom{2k-1}{i}}{\binom{2k-2}{j}} \\
&= \frac{k}{2k-1} \sum_{i=0}^{k-1}\binom{2k-1}{i} \left(2\sum_{j=i}^{k-2}\frac{1}{\binom{2k-2}{j}} + \frac{1}{\binom{2k-2}{k-1}}\right) \\
&= 2k \left( \frac{1}{2k-1} \sum_{j=0}^{k-2}\sum_{i=0}^{j} \frac{\binom{2k-1}{i}}{\binom{2k-2}{j}} + \frac{2^{2k-2}}{k \binom{2k}{k}} \right) = k \sum_{i=0}^{k-1} \frac{1}{2i+1}, 
\end{align*}
where the last step follows from Equation \eqref{eq_doublesum2} in Lemma~\ref{lemma_doublesum}. This proves Equation \eqref{eq_t0_growth}. 
\end{proof}

According to Equation \eqref{eq_t0_growth} we can express $T(k,k)$ as $k$ times the sum of the first $k$ ``odd'' terms in the harmonic series. Since it is well known that the harmonic series grows logarithmically, we get the asymptotic formula in Corollary~\ref{cor_T(k,k)}: 

\begin{proof}[Proof of Corollary~\ref{cor_T(k,k)}]
Since the harmonic numbers $H_n=1+\frac{1}{2}+\dots+\frac{1}{n}$ satisfy
\begin{equation*}
H_n = \ln(n) + \gamma + \frac{1}{2n} + O\left(\frac{1}{n^2}\right),
\end{equation*}
where $\gamma$ is the Euler--Mascheroni constant, it follows from \eqref{eq_t0_growth} that
\begin{align}\label{eq_Tkk_bigO}
T(k,k)&=k\left(H_{2k}-\frac{1}{2}H_k\right) \nonumber\\
&= k\left(\frac{\ln k}{2} + \ln{2} + \frac{\gamma}{2} \right) + O\left(\frac{1}{k}\right),
\end{align}
which proves Equation \eqref{eq_Tkk_growth}. 
\end{proof}

\subsection{A\nbd controlled M\nbd process; proofs of Proposition~\ref{time_equations}, Theorem~\ref{thm_tk_growth}, Corollary~\ref{cor_asym} and Proposition~\ref{tw<b_growth}}\label{proofs_AM}

We now turn to the A\nbd controlled M\nbd process. The formulas in Proposition~\ref{prop_Williams} are not proved here, since they are already included in \cite{Williams}, but the proof is easily done in the same way as for Proposition~\ref{time_equations} below. 

Note that under Policy~A Equation \eqref{eq_ta3} is valid only when $w<b$, whereas for $w\geq b$ we instead have that $T^A(w,b) = T^A(w-1,b)$, so there are two separate cases to consider depending on the starting point. If $w\geq b$ we can write $T^A(w,b) = T^A(b,b) = t_b$. On the other hand, if $w<b$ then no balls are removed as long as there are more black ones, and so the value of $T^A(w,b)$ can be determined using Lemma~\ref{lemma_method} with respect to the boundary value $t_k$, where $k= \left\lfloor \frac{w+b}{2} \right\rfloor$. There are two different formulas given in Proposition~\ref{time_equations}, depending on whether the initial total number of balls is even or odd. 

\begin{proof}[Proof of Proposition~\ref{time_equations}]
The results follow from Lemma~\ref{lemma_method}. First suppose that the total number of balls is even. For $k=(w+b)/2$ there is some $c\in\{1,2,\dotsc,k-1\}$ so that $w=k-c$ and $b=k+c$. The recursion 
\begin{equation*}
T^A(k-c,k+c) = 1 + \frac{k-c}{2k} T^A(k-(c-1),k+c-1) + \frac{k+c}{2k} T^A(k-(c+1),k+c+1)
\end{equation*}
holds for all $0<c<k$, with the boundary values being $T^A(0,2k)=0$ and $T^A(k,k) = t_k$. From Lemma~\ref{lemma_method} it follows that for all $0\leq c \leq k$
\begin{equation*}
T^A(k-c,k+c) = t_k + (Q_1(k)-t_k) 2^{-(2k-2)} \sum_{j=k}^{k+c-1}\binom{2k-1}{j} - Q_1(c),
\end{equation*}
where the function $Q_1(c)$ is given by
\begin{align*}
Q_1(c) &= \sum_{i=2}^{c} \sum_{j=1}^{i-1} \frac{2k}{k+j} \prod_{l=j+1}^{i-1} \frac{k-l}{k+l} 
= 2k \sum_{i=2}^{c} \sum_{j=1}^{i-1} \frac{\binom{2k-1}{k+i-1}}{(2k-1)\binom{2k-2}{k+j-1}} \\
&= \frac{2k}{2k-1} \sum_{i=k+1}^{k+c-1} \sum_{j=k}^{i-1} \frac{\binom{2k-1}{i}}{\binom{2k-2}{j}} = 2k\alpha_k(c), 
\end{align*}
with $\alpha_k(n)$ defined in \eqref{eq_alphabeta}. Similarly, for an odd number of balls the recursion is 
\begin{equation*}
T^A(k+1-c,k+c) = 1 + \frac{k+1-c}{2k+1} T^A(k+2-c,k+c-1) + \frac{k+c}{2k+1} T^A(k-c,k+c+1)
\end{equation*}
where $0<c<k+1$ and the boundary values are $T^A(0,2k+1) = 0$ and $T^A(k+1,k) = t_k$. In this case Lemma~\ref{lemma_method} gives 
\begin{equation*}
T^A(k+1-c,k+c) = t_k + (Q_2(k+1)-t_k) \frac{2^{-(2k-1)}}{1+p_k} \; \sum_{i=k}^{k+c-1}\binom{2k}{i} - Q_2(c)
\end{equation*}
for $0\leq c \leq k+1$, with
\begin{align*}
Q_2(c) &= \sum_{i=2}^{c} \sum_{j=1}^{i-1} \frac{2k+1}{k+j} \prod_{l=j+1}^{i-1} \frac{k-l+1}{k+l} 
= (2k+1) \sum_{i=2}^{c} \sum_{j=1}^{i-1} \frac{\binom{2k}{k+i-1}}{(2k-1)\binom{2k-1}{k+j-1}} \\
&= \frac{2k+1}{2k} \sum_{i=k+1}^{k+c-1} \sum_{j=k}^{i-1} \frac{\binom{2k}{i}}{\binom{2k-1}{j}} = (2k+1)\beta_k(c),
\end{align*}
where $\beta_k(c)$ is defined in \eqref{eq_alphabeta}. This proves \eqref{eq_ time1} and \eqref{eq_ time2}. These formulas make it possible to calculate $T^A(w,b)$ for arbitrary $w$ and $b$ in terms of $t_k$. We show now that $t_k$ satisfy the recursive equation in \eqref{eq_time_recursion}. Since $T^A(k+1,k+1) = T^A(k,k+1)$ and $\beta_k(1)=0$ for all $k$ we get by inserting the value $c=1$ into \eqref{eq_ time2} that
\begin{equation*}
t_{k+1} = \frac{1-p_k}{1+p_k}t_k + (2k+1)\beta_k \frac{2p_k}{1+p_k}, 
\end{equation*}
where we use the notation $\beta_k := \beta_k(k+1)$. Note that compared to \eqref{eq_vk_recursion} the only difference is the factor $\beta_k$, which according to \eqref{eq_doublesum1} is equal to
\begin{equation}\label{eq_betak}
\beta_k := \beta_k(k+1) = \frac{1}{2k} \sum_{i=k+1}^{2k} \sum_{j=k}^{i-1} \frac{\binom{2k}{i}}{\binom{2k-1}{j}} = \frac{1}{2k} \sum_{j=0}^{k-1} \sum_{i=0}^{j} \frac{\binom{2k}{i}}{\binom{2k-1}{j}} = \frac{1}{2}\sum_{i=0}^{k-1} \frac{1}{2i+1},
\end{equation}
proving the last part of the proposition. 
\end{proof}

Using the results in Proposition~\ref{time_equations} the value of $T^A(w,b)$ can be exactly determined for arbitrary values of $w$ and $b$. For large numbers, however, the calculations get demanding and an asymptotic expression might be more useful. We therefore turn to the asymptotics in the A\nbd controlled M\nbd process and first give a proof for Theorem~\ref{thm_tk_growth}, which contains an asymptotic expression for $T^A(k,k)$ up to constant order term. 

\begin{proof}[Proof of Theorem~\ref{thm_tk_growth}]
Recall the notation $t_k:=T^A(k,k)$ and further define the function $\zeta_k$ by
\begin{equation*}
\zeta_k := t_k-\left(2k-\frac{1}{p_k}+\frac{\pi}{4}\right)\beta_k - \frac{3\pi}{16}+\frac{1}{4p_k}+\frac{1}{4}. 
\end{equation*}
with $p_k$ as in \eqref{eq_pk} and $\beta_k$ as in \eqref{eq_betak}. A recursion for $t_k$ is given in \eqref{eq_time_recursion}. Notice also that
\begin{equation*}
\beta_{k+1} = \beta_k + \frac{1}{4k+2}, \quad p_{k+1}=\frac{2k+1}{2k+2}p_k.
\end{equation*}
With the help of these equations $\zeta_{k+1}$ can be expressed recursively as
\begin{align}
\zeta_{k+1} &= t_{k+1}-\left(2k+2-\frac{1}{p_{k+1}}+\frac{\pi}{4}\right)\beta_{k+1} - \frac{3\pi}{16}+\frac{1}{4p_{k+1}}+\frac{1}{4} \nonumber\\[0.5em]
&= \frac{1-p_k}{1+p_k}t_k + (2k+1)\beta_k \frac{2p_k}{1+p_k} - \left(2k+2-\frac{1}{p_{k+1}}+\frac{\pi}{4}\right)\left(\beta_k + \frac{1}{4k+2}\right) \nonumber\\
&\hspace{3em} - \frac{3\pi}{16}+\frac{1}{4p_{k+1}}+\frac{1}{4} \nonumber\\[0.8em]
&= \frac{1-p_k}{1+p_k}\left( t_k-\left(2k-\frac{1}{p_k}+\frac{\pi}{4}\right)\beta_k - \frac{3\pi}{16}+\frac{1}{4p_k}+\frac{1}{4} \right) \nonumber\\
&\hspace{3em} + \beta_k\left( \frac{2p_k}{1+p_k}\left(1-\frac{\pi}{4}\right) -2+\frac{1}{p_{k+1}}-\frac{1-p_k}{p_k(1+p_k)} \right) - \frac{1-p_k}{4p_k(1+p_k)} \nonumber\\
&\hspace{3em} - \frac{1}{4k+2} \left(2k+2-\frac{1}{p_{k+1}}+\frac{\pi}{4}\right) + \frac{1}{4p_{k+1}} + \frac{2p_k}{1+p_k}\left( \frac{1}{4} - \frac{3\pi}{16} \right) \nonumber\\[0.8em]
&= \frac{1-p_k}{1+p_k}\zeta_k + \frac{2p_k}{1+p_k} Z_k, \label{eq_Bcdefg}
\end{align}
where in the last step we have introduced $Z_k := \beta_kc_k + d_k + e_k + f_k + g_k$, with
\begin{align*}
&c_k := \frac{p_k-p_{k+1}}{2p_k^2 p_{k+1}}-\frac{\pi}{4}, \quad
d_k :=  \frac{(1-p_k)(p_k-p_{k+1})}{8p_k^2 p_{k+1}} - \frac{\pi}{16}, \\
&e_k :=  \frac{1+p_k}{4p_k p_{k+1}(2k+1)} - \frac{\pi}{8}, \quad
f_k :=  \frac{1}{4}-\frac{p_k}{4p_{k+1}}, \quad
g_k :=  - \frac{\pi (1+p_k)}{16(2k+1)p_k}.
\end{align*}
The key thing is that $\beta_k c_k, d_k,e_k,f_k$ and $g_k$ all tend to zero as $k\to\infty$. Indeed, using Stirling's approximation
\begin{equation*}
n! = \sqrt{2\pi n} \left( \frac{n}{e} \right)^n \left( 1 + O\left( \frac{1}{n}\right) \right)
\end{equation*}
we get that
\begin{equation}\label{eq_pk_approx}
p_k = \frac{1}{\sqrt{\pi k}} \left( 1 + O\left( \frac{1}{k}\right) \right)
\end{equation}
and from this follows that $c_k, f_k = O(1/k)$ while $d_k, e_k, g_k = O(1/\sqrt{k})$. Furthermore, from \eqref{eq_Tkk_growth} and $T(k,k)=2k\beta_k$ we know that $\beta_k = O(\ln(k))$, so it follows that $\beta_kc_k = O(\ln(k)/k)$. This means that $Z_k\to 0$ as $k\to\infty$. Now let $\theta_k := (1-p_k)/(1+p_k)$ and rewrite \eqref{eq_Bcdefg} as
\begin{equation*}
\zeta_{k+1} = \theta_k\zeta_k + (1-\theta_k) Z_k.
\end{equation*}
For any $\varepsilon>0$ one can find $N$ such that $|Z_k|<\varepsilon$ for every $k\geq N$. By induction on $k\geq N$, we now get that
\begin{equation} \label{ineq_zeta}
|\zeta_{k+1}| \leq \theta_k \theta_{k-1}\dotsm \theta_N |\zeta_N| + \varepsilon.
\end{equation}
According to Stirling's formula $1-\theta_k \sim 2/(1+\sqrt{\pi k}) \geq 1/\sqrt{\pi k}$, which means that the infinite series $\sum (1-\theta_k) = \infty$. By the comparison test, this implies that also the series $\sum \ln \theta_k$ diverges, meaning in turn that $\prod \theta_k=0$. Thus, from Equation \eqref{ineq_zeta} it follows that $\lim\sup |\zeta_{k+1}| \leq \varepsilon$ when $k\to\infty$. Since this holds for any arbitrarily small $\varepsilon$ we have that $\zeta_k \to 0$. 

Finally, since $T(k,k)=2k\beta_k$ it follows that we can replace $\beta_k$ with $\frac{1}{2k}$ times the expression in \eqref{eq_Tkk_bigO}, and similarly we can replace $p_k^{-1}$ using Equation \eqref{eq_pk_approx}. Making these insertions in the expression for $\zeta_k$ gives the result in \eqref{eq_tk_growth}. 
\end{proof}

\begin{proof}[Proof of Corollary~\ref{cor_asym}]
The result follows immediately when comparing equations \eqref{eq_V0(k,k)} and \eqref{eq_vk_approx} as well as equations \eqref{eq_t0_growth} and \eqref{eq_tk_growth}. 
\end{proof}

From Corollary~\ref{cor_asym} it is easy to see how the asymptotic behavior differs in the M\nbd process and the A\nbd controlled M\nbd process when starting with an equal number of white and black balls. When the ratio of black balls in the initial state is greater than $\frac{1}{2}$ the asymptotic behavior is different, as stated in Proposition \ref{tw<b_growth}. 

\begin{proof}[Proof of Proposition~\ref{tw<b_growth}]
Assume that the A\nbd controlled M\nbd process is started from $b=\lceil xN \rceil$ and $w=\lfloor (1-x)N \rfloor$ such that $b>w>0$, i.e. $\frac{1}{2}<x<1$. First we prove that the M\nbd process and the A\nbd controlled process are a.s. equal when $N\to\infty$. Note that the two processes differ only if $W_n\geq B_n$ for some $n$. Let
\begin{equation*}
G:=\min\{n:W_n\geq B_n\}
\end{equation*}
be the first time this event occurs, with $G=\infty$ if it never happens. Note that since $w<b$ we always have that $G>0$. We want to find the probability that $G$ is finite. 

In view of Equation \eqref{eq_P}, we get for all $w<b$ that the probability of getting no final black balls is
\begin{equation*}
\PP(B_H=0) = \frac{1}{2^{N-1}} \sum_{i=0}^{w-1}\binom{N-1}{i} \leq \frac{w}{2^{N-1}}\binom{N-1}{w-1} < \frac{w}{2^{N}}\binom{N}{w}.
\end{equation*}
Using Stirling's approximation and the fact that $w\sim (1-x)N$ we now get that
\begin{equation*}
\PP(B_H=0) \sim \sqrt{\frac{1}{2\pi} \left( \frac{1}{x}-1\right)}\cdot \frac{\sqrt{N}}{(2x^x(1-x)^{1-x})^N}.
\end{equation*}
On the interval $0<x<1$ the function $2x^x(1-x)^{1-x}$ is real valued and attains a unique local minimum in $x=\frac{1}{2}$, for which the value is 1. Since we have assumed that $\frac{1}{2}<x<1$ the base in the denominator is strictly greater than 1, so $\PP(B_H=0)\to 0$ as $N\to\infty$. 

By the law of total probability we have that
\begin{equation*}
\PP(B_H=0) = \sum_{n=1}^\infty \PP(B_H=0 | G=n)\PP(G=n) + \PP(B_H=0 | G=\infty)\PP(G=\infty).
\end{equation*}
If $G=\infty$ then there are never more white balls than black, so $\PP(B_H=0 | G=\infty)=0$. On the other hand, if $G=n<\infty$ we can use the Markov property and restart the M\nbd process from $(W_n,B_n)$, and since $W_n\geq B_n$ we have that $\PP(B_H=0 | G=n) \geq \frac{1}{2}$ (see also Equation \eqref{eq_P}). This means that
\begin{equation*}
\PP(B_H=0) \geq \frac{1}{2}\sum_{n=1}^\infty \PP(G=n) = \frac{\PP(G<\infty)}{2},
\end{equation*}
which in turn means that also $\PP(G<\infty)\to 0$ as $N\to\infty$. 

Thus, when $N\to\infty$, almost surely $W_n< B_n$ for all $n$, so no balls are removed even if Policy~A is applied throughout the process. This means that there will (almost surely) be no difference between the M\nbd process and the A\nbd controlled M\nbd process. Naturally we then have that $V(w,b)\sim V^A(w,b)$ and $T(w,b)\sim T^A(w,b)$. 

Since $w\leq\left\lfloor\tfrac{N}{2}\right\rfloor<b$ we clearly get from \eqref{eq_vk_approx} that 
\begin{equation*}
V^A(w,b)\geq V^A\left(\left\lfloor\tfrac{N}{2}\right\rfloor,\left\lfloor\tfrac{N}{2}\right\rfloor\right) \sim N,
\end{equation*}
and obviously $V^A(w,b)\leq N$, so Equation \eqref{eq_Vw<b} holds. Equation \eqref{eq_Tw<b} follows from 
\begin{equation*}
T(w,b) \sim \frac{N}{2} \ln\left(\frac{1}{1-2(1-x)}\right),
\end{equation*}
which is proved in \cite{Flajolet} Theorem~3 and Note~3(iii). 
\end{proof}

\subsection{Conditional M\nbd process; proof of Proposition \ref{Tstar}}\label{proofs_CM}

An explicit expression for $T_*(w,b)$ can be found using Lemma~\ref{lemma_method}, which gives an expression rather similar to Equation \eqref{eq_T0} but more complicated. 

\begin{proof}[Proof of Proposition \ref{Tstar}]
Note that 
\begin{equation*}
T(w,b) = T_*(w,b)h(b) + T_*(b,w)(1-h(b)),
\end{equation*}
so when $w=b=k$ we get that 
\begin{equation*}
T(k,k) = T_*(k,k)h(k) + T_*(k,k)(1-h(k)) = T_*(k,k).
\end{equation*}
Furthermore, when $b=\lceil xN \rceil$ with $x>\frac{1}{2}$ we have seen in the proof of Proposition~\ref{tw<b_growth} that $1-h(b)=\PP(B_H=0)\to0$, which means that
\begin{equation*}
T(w,b)-T_*(w,b) = (1-h(b))(T_*(b,w)-T_*(w,b)) \to 0,
\end{equation*}
completing the proof. 
\end{proof}


\acks

I am very grateful to Prof. Paavo Salminen for encouraging me to work with the Mabinogion urn model, as well as for many helpful discussions. I am also grateful to Prof. James Wan and Prof. Christophe Vignat for the collaboration, and in particular to Prof. Wan for providing a proof of Lemma~\ref{lemma_doublesum}. I also want to thank the anonymous referees for their valuable comments and suggestions for improvements.

\end{document}